\documentclass[a4paper, 11pt]{article}
\usepackage{authblk, setspace, subfig, caption}
\usepackage{url}
\usepackage{amsmath}

\usepackage{amssymb,amscd} 
\usepackage{enumerate} 
\usepackage[active]{srcltx} 
\usepackage{amsthm, leftidx}

\usepackage{multirow, booktabs}
\usepackage{algorithm}
\usepackage{algpseudocode}

\usepackage{graphicx}

\usepackage{authblk}

\usepackage{amsthm} 

\usepackage{tikz}
\usetikzlibrary{mindmap,shadows, calc}
\usepackage[hidelinks,pdfencoding=auto]{hyperref}

\theoremstyle{definition}

\newtheorem{convention}{Convention}
\newtheorem{rem}{Remark}
\newtheorem{ex}{Example}

\theoremstyle{plain}

\newtheorem{prop}{Proposition}
\newtheorem{cor}{Corollary}

\newcommand{\RR}{\mathbb{R}}

\def\states{\mathcal X}

\def\cset{\mathcal M}
\def\csete{\mathcal C}
\def\allgambles{\mathcal L}

\def\lpr{\underline P}

\def\extp{\mathcal E}
\def\face{\mathcal M}

\usepackage[square,sort,comma,numbers]{natbib} 
\usepackage{algorithm}
\usepackage{algpseudocode}

\usepackage{hyperref}

\newcommand{\low}[1]{{\underline{#1}}}

\title{Errors bounds for finite approximations of coherent lower previsions on finite probability spaces}
\author[1]{Damjan \v{S}kulj \\ University of Ljubljana, Faculty of Social Sciences \\ Kardeljeva pl. 5, SI-1000 Ljubljana, Slovenia \\ \href{mailto:damjan.skulj@fdv.uni-lj.si}{\tt damjan.skulj@fdv.uni-lj.si} }

\begin{document}
	
	\maketitle

	\begin{abstract}
Coherent lower previsions are general probabilistic models allowing incompletely specified probability distributions. 
However, for complete description of a coherent lower prevision -- even on finite underlying sample spaces -- an infinite number of assessments is needed in general. 
Therefore, they are often only described approximately by some less general models, such as coherent lower probabilities or in terms of some other finite set of constraints. 
The magnitude of error induced by the approximations has often been neglected in the literature, despite the fact that it can be significant, with substantial impact on consequent decisions.
An apparent reason is that no widely used general method for estimating the error seems to be available at the moment. 

This paper provides a practically applicable method that allows calculating an upper bound for the maximal error induced by approximating a coherent lower probability with its values on a finite set of gambles. 
An algorithm is also provided with an estimation of its computational complexity.

		
	\end{abstract}
	
\section{Introduction}
Models of \emph{imprecise probabilities} have been developed to cope with uncertainty in probability distributions. 
Single precise models are thus replaced by models compatible with multiple (precise) probability distributions. 
The advantage of such models compared to the classical precise models is that they can incorporate a lower or higher degree of uncertainty, which is represented through larger or smaller sets of compatible probability distributions.
Thus imprecise models subsume precise models as one extreme as well as the models of complete uncertainty on the other.
A review of models and applications of imprecise probabilities can be found in \cite{augustin2014introduction}. 

One of the most popular and also most general models of imprecise probabilities are \emph{coherent lower previsions} (see e.g. \cite{miranda2008, TroffaesDeCooman2014}). 
A coherent lower prevision $\lpr$, in general given on a measurable space $(\states, \mathcal A)$, is an imprecise probability model based on judgements about the lower or upper expectations of a set of random variables $\mathcal K$, also called \emph{gambles}. 
The judgement $\lpr(f) = a$ implies that every precise probability distribution $P$ compatible with $\lpr$ must satisfy $E_P(f) \ge a$, that is $\lpr(f)$ means that the expectation of $f$ is at least $a$. 
\emph{Coherence} in this context means that the judgements on the set of gambles allow, for every gamble $f$, the existence of at least one precise probability distribution $P$ compatible with $\lpr$ for which $E_P(f)=\lpr(f)$. 

A coherent lower prevision $\lpr$ specified on a set of gambles $\mathcal K$ can have multiple possible extensions to a larger set, say $\mathcal H\supset \mathcal K$.
In other words, there can be multiple coherent lower previsions that coincide on a set of gambles.
In particular, a coherent lower prevision may be approximated by a more specific model, such as \emph{coherent lower probability} (see e.g.\cite{antonucci10a}), in which case its restriction to \emph{indicator gambles} is only known, i.e. an \emph{indicator gamble} $1_A$ is a map $\states \to \RR$ such that $1_A(x)$ equals 1 if $x\in A$ and 0 otherwise. 
There are variety of reasons for approximating coherent lower probabilities with less general methods. 
One reason is that in general there is no nice or elegant way to represent a general coherent lower prevision. 
Lower and upper probabilities can be much more elegant and intuitive as approximations. 
We must also keep in mind that coherent lower previsions, even on finite spaces, in general cannot be represented in terms of a single function or any other reasonably sized collection of values. 
In best case they can be represented as sets of extreme points of their credal sets, which in most cases are very large. 
Instead of calculating all extreme points it is sometimes more convenient to approximate a coherent lower prevision with its values on a suitable set of gambles and apply the \emph{natural extension} for further calculations. 
In many cases, computations with coherent lower previsions are computationally intensive (consider for instance calculating lower prevision corresponding to imprecise Markov chains \cite{decooman-2008-a, skulj:09IJAR, skulj:15AMC}), which makes it reasonable to keep the set $\mathcal K$ of moderate size. 
 
In this paper we analyze the following problem. 
Let $\lpr$ be a lower prevision on a finite sample space $\states$. 
Its full description would in general require detailed information of its credal set, whose set of extreme points can be very large. 
Suppose that instead we know the values of $\lpr$ on a set of gambles $\mathcal K$. 
The restriction $\lpr_{\mathcal K}$ approximates $\lpr$ and the natural question arises, how good is this approximation. 
Given the restriction, $\lpr$ is its extension, but there might be other extensions too. 
Therefore, we would like to know by how much can another extension deviate from $\lpr$. 
In other words, we want to find the maximal distance between two arbitrary extensions of a coherent lower prevision to the set of all gambles. 

In our analysis of the maximal possible distance we first show that the maximal possible distance is always reached when one of the extensions is the \emph{natural extension}. 
Consequently, much of the analysis is done on the credal set of the natural extension with the special emphasis on its extreme points. 
Our main theoretical result gives an upper bound for the maximal distance in terms of distances between the extreme points. 
We also provide an algorithm for finding the maximal possible distance and estimate its computational complexity.

The paper is structured as follows. In Section~\ref{s-nbr} we review basic concepts of imprecise probabilities with the emphasis on coherent lower previsions. In Section~\ref{s-cacs} we analyze basic properties of credal sets as convex polyhedra and apply some general concepts of convex analysis to the case of credal sets. Our main theoretical results are in Section~\ref{s-dbclp}. The algorithm for calculating the maximal possible distance is described in Section~\ref{s-a}. The paper concludes with Section~\ref{s-c}. 

\section{Notation and basic results}\label{s-nbr}
In this section we introduce the notation and review the concepts used in the paper. When possible we will stick with the standard terminology used in the theory of imprecise probabilities, which will sometimes be supplemented by the standard terminology of convex analysis, linear algebra and optimization. 

The object of our analysis will be \emph{coherent lower previsions} which are one of the most general models used to represent partially specified probabilities. They encompass several particular models, such as \emph{coherent lower and upper probabilities, 2- and n-monotone capacities, belief and plausibility functions, lower expectation functionals} and others. Mathematically, coherent lower previsions are superlinear functionals that can be equivalently represented as lower envelopes of expectation functionals. 

\subsubsection*{Gambles}
Throughout this paper let $\states$ represent a finite set, a \emph{sample space}, and $\allgambles$ the set of all real-valued maps on $\states$, also called \emph{gambles}. Equivalently, $\allgambles$ may be viewed as the set of vectors in $\RR^{|\states|}$. By $1_A$ we will denote the \emph{indicator gamble} of a set $A\subseteq \states$:
\begin{equation}
1_A(x) = 
\begin{cases}
1 & x \in A \\
0 & \text{otherwise}. 
\end{cases} 	
\end{equation}
We will write $1_x$ instead of $1_{\{x \} }$ for elements $x\in \states$. 

The set of gambles will be endowed by the standard inner product 
\begin{equation}\label{eq-inner-product}
	f\cdot g = \sum_{x\in \states} f(x)g(x),
\end{equation}
which generates the $l^2$ norm:
\begin{equation}
\| f\| = \sqrt{f\cdot f} = \sqrt{\sum_{x\in\states}f(x)^2},
\end{equation}
and the Euclidean distance between vectors:
\begin{equation}
d(f, g) = \| f-g \|,
\end{equation}
which will be used by default throughout the paper. 

\subsubsection*{Linear previsions} A \emph{linear prevision} $P$ is an expectation functional with respect to some probability mass vector $p$ on $\states$. It maps a gamble $f$ into a real number $P(f)$.
Usually, we will write 
\begin{equation}
P(f) = \sum_{x\in\states} p(x)f(x) =: P\cdot f.
\end{equation} 
The set of linear previsions is therefore a subset of the dual space of $\allgambles$. 

The inner product notation used on the right hand side of the above equation is introduced because we will often use linear functionals of the form $
f\mapsto p\cdot f
$
where the vector $p$ will not necessarily be a probability mass vector. We will then use the inner product notation to avoid misinterpretations. Without danger of confusion we will therefore interpret a linear prevision $P$ as a vector with the same length as gambles in $\allgambles$. 

\subsubsection*{Probability simplex} If the sample space $\states$ contains exactly three elements, say $\states = \{x, y, z \}$, the probability mass vectors can be represented as points of the form $(p(x), p(y), p(z))$ in $\RR^3$. However, since the restriction $p(x)+p(y)+p(z)=1$ applies, they in fact form a two dimensional space, which can be depicted as an equilateral triangle with vertices $x, y$ and $z$. Given any point in this triangle, the sum of distances to its sides is constantly equal to its altitude, which equals $\frac{\sqrt{3}}{2}a$, where $a$ is the common length of the sides. Taking $a=\frac{2}{\sqrt{3}}$ makes the altitude equal to 1. The distance of a point from each side now denotes the probability of the point in the opposite vertex (see Fig. \ref{fig-psimplex}).
%
%
Probability simplex diagrams are very useful to illustrate concepts of imprecise probabilities; however, one needs to be cautious not to be mislead by specifics of low dimensional probability spaces. 

\subsubsection*{Coherent lower previsions} 
A \emph{coherent lower prevision} on an arbitrary set of gambles $\mathcal K$ is a mapping $\lpr\colon \mathcal K\to \RR$ that allows the representation 
\begin{equation}
\lpr (f) = \min_{P\in\cset(\lpr)} P(f)
\end{equation}
for every $f\in \mathcal K$, where $\cset(\lpr)$ is a closed and convex set of linear previsions. The set $\cset(\lpr)$ is called a \emph{credal set} of $\low P$. We will often denote a credal set just by $\cset$. 
\begin{figure}
	\centering
	\begin{minipage}[t]{0.45\textwidth}
		\centering
		\begin{tikzpicture}[scale = 4]
		\coordinate (x) at (0.000,  0.000);
		\coordinate (z) at (0.577,  1.000);
		\coordinate (y) at (1.155,  0.000);
		
		\draw[gray] (x) -- (y) -- (z) --cycle;
		\node[below left] at (x) {$x$}; 
		\node[below right] at (y) {$y$};
		\node[above] at (z) {$z$};
		
		\coordinate (p) at (0.4330 ,   0.2500);
		
		\draw[thin, dashed] (p)--(0.4330,  0) node[midway, right] {\tiny $P(z)$}; 
		\draw[thin, dashed] (p)--(0.2165,   0.3750) node[midway, below left] {\tiny $P(y)$};
		\draw[thin, dashed] (p)--(0.8660,    0.5000) node[midway, above] {\tiny $P(x)$};
		
		\draw[fill=white] (p) circle (.3pt) node[color = black, right] {\tiny $P = (1/2, 1/4, 1/4)$};
		\end{tikzpicture}
		\caption{Probability simplex: the distance from a side denotes the probability of the element at the opposite vertex.}\label{fig-psimplex}
	\end{minipage}\hfill
	\begin{minipage}[t]{0.45\textwidth}
		\centering
		\begin{tikzpicture}[scale = 4]
		\coordinate (x) at (0.000,  0.000);
		\coordinate (z) at (0.577,  1.000);
		\coordinate (y) at (1.155,  0.000);
		
		\draw[gray] (x) -- (y) -- (z) --cycle;
		\node[below left] at (x) {$x$}; 
		\node[below right] at (y) {$y$};
		\node[above] at (z) {$z$};
		
		\coordinate (E1) at (0.531,  0.280);
		\coordinate (E2) at  (0.531,  0.227);
		\coordinate (E3) at  (0.592,  0.191);
		\coordinate (E4) at (0.670,  0.202);
		\coordinate (E5) at (0.702,  0.477);
		
		\draw[gray!40!white, fill = gray!30!white] (E1) -- (E2)--(E3)	--	(E4) --	(E5)  -- cycle; 
		\node at ($(E1)+(.07, 0)$) {\tiny $\cset$};
		\end{tikzpicture}	
		\caption{The credal set $\cset$ of the lower prevision from Example~\ref{ex-lpr1}.} \label{fig-credal-set}
	\end{minipage}
\end{figure}

\subsubsection*{The natural extension} Given a coherent lower prevision $\lpr$ on $\mathcal K$, it is possible to extend it to the set of all gambles $\allgambles$ in possibly several different ways. However, there is a unique minimal extension, called the \emph{natural extension}:
\begin{equation}
\low E(f) = \min_{P\in\cset(\lpr)} P(f). 
\end{equation} 
As the natural extension is the \emph{lower envelope} or the \emph{support function} of a credal set, containing expectation functionals, we may call a coherent lower prevision defined on the entire $\allgambles$ a \emph{lower expectation functional}. 

A mapping $\lpr \colon \mathcal K\to \RR$, where $\mathcal K$ is a linear (vector) space, is a coherent lower prevision if and only if it satisfies the following axioms (\cite{miranda2008}) for all $f, g\in \mathcal K$ and $\lambda \ge 0$: 
\begin{itemize}
	\item[(P1)] $\lpr (f) \ge \inf_{x\in \states} f(x)$ [accepting sure gains];
	\item[(P2)] $\lpr (\lambda f) = \lambda \lpr (f)$ [positive homogeneity];
	\item[(P3)] $\lpr (f + g) \ge \lpr(f) + \lpr(g)$ [superlinearity].
\end{itemize}
An easy consequence of the definitions is \emph{constant additivity}:
\begin{equation}
	\lpr (f + \lambda 1_\states) = \lpr(f) + \lambda
\end{equation}
for any $\lambda\in \RR$. 

\section{Convex analysis on credal sets}\label{s-cacs}
\subsection{Credal set as a closed convex set} 
A credal set is a closed and convex set of linear previsions. Since every linear prevision can be uniquely represented as a probability mass vector, a credal set can be represented as a convex set of probability mass vectors. The set $\cset$ is therefore the maximal set of $|\states|$-dimensional vectors $p$ satisfying 
\begin{align}
p\cdot f & \ge \lpr(f) & & \text{ for every } f\in \mathcal K, \label{eq-cset-constraints-first}\\
p\cdot 1_x & \ge 0 & & \text{ for every } x\in \states\label{eq-cset-constraints-nonnegativity} \text{ and } \\
p\cdot 1_\states & = 1. \label{eq-constraint-equality} 	
\end{align}	
In the case where $\mathcal K$ is finite, $\cset$ is bounded by a finite number of support hyperplanes, and 
%
is therefore a \emph{convex polyhedron}. Strictly speaking, it is an $\mathcal H$-polyhedron, which means that it is bounded and an intersection of a finite number of half spaces. According to Theorem~14.3 in \cite{gruber:07CDG} every $\mathcal H$-polyhedron in an $\RR^m$ is also a $\mathcal V$-polyhedron, which means that it is a convex combination of a finite number of extreme points. 

From now on we will call credal sets that are convex polyhedra \emph{finitely generated credal sets}. Similarly, we will denote coherent lower previsions defined on finite sets of gambles or their natural extensions \emph{finitely generated coherent lower previsions}.
\begin{ex}\label{ex-lpr1}
	Let $\lpr$ be a lower prevision on $\mathcal K = \{ f_1, \ldots, f_5 \}$ where 
	\begin{align*}
		f_1 & = (0,	1,	0.5) &	f_2 & = (0,	0.5, 1) \\
		f_3 & = (0.15,	0, 1) & 	f_4 & = (1,	0,	0.6) \\
		f_5 & = (0.2,	1,	0) \\
	\end{align*}
	and 
	\begin{align*}
		\lpr(f_1) & = 0.46 & \lpr(f_2) & = 0.4 & \lpr(f_3) & = 0.25 \\
		\lpr(f_4) & = 0.44 & \lpr(f_5) & = 0.4
	\end{align*}
	The credal set corresponding to $\lpr$ is depicted in Figure~\ref{fig-credal-set}.
\end{ex}
According to the above, it would be suitable to extend the domain of $\lpr$ with the gambles of the form $1_x$ for every $x\in\states$. Doing so, though, may result in a non-coherent lower prevision, because other constraints my already imply that $\lpr(1_x) \ge 0$, where the inequality may even be strict. Therefore we adopt the following convention: 
\begin{convention}\label{conv-include-1x}
	The domain $\mathcal K$ of all lower previsions used will contain all gambles of the form $1_x$ together with the value $\lpr(1_x)=0$, unless $\lpr(1_x)\ge 0$ is already implied by other values of $\lpr$ on $\mathcal K$. 
\end{convention}
Assuming the above convention, the credal set of a coherent lower prevision $\lpr$ is the set of vectors $p$ satisfying constraints \eqref{eq-cset-constraints-first} and \eqref{eq-constraint-equality}. 

\subsubsection*{Faces and extreme points of a finitely generated credal set}
The \emph{faces} of a credal set $\cset$ are the sets of the form
\begin{equation}
	\cset_f = \{ P\in\cset \colon P(f) = \low E(f)  \},
\end{equation}
where $f$ is an arbitrary gamble. 
The smallest faces are exactly the extreme points and the faces of codimension 1 are called \emph{facets}\footnote{The codimension 1 is meant relative to the dimension of $\cset$. That is $\dim{\face_f} = \dim \cset-1$. Note also that a credal set is at most of dimension $|\states|-1$ because of the constraint $P(1_\states) = 1$}. 
The set of all extreme points of $\cset$ will be denoted by $\extp(\cset)$ or simply $\extp$. The set of extreme points of a face $\cset_f$ will be denoted by $\extp_f$, and $\extp_f\subseteq \extp$ holds. 
\begin{ex}
	The extreme points of the credal set from Example~\ref{ex-lpr1} are 
	\begin{align*}
		E_1 & = (0.4,	0.32, 	0.28) & E_2 = (0.43,	0.35,	0.23) \\
		E_3 & = (0.39,	0.42,	0.19) & E_4 = (0.32,	0.48,	0.20)  \\
		E_5 & = (0.15,	0.37,	0.48	) 
	\end{align*}
	(See Figure~\ref{fig-normal-cones}.)
\end{ex}
\begin{figure}
	\centering
	\begin{minipage}[t]{0.45\textwidth}
		\centering
		\begin{tikzpicture}[scale = 4]
		\coordinate (x) at (0.000,  0.000);
		\coordinate (z) at (0.577,  1.000);
		\coordinate (y) at (1.155,  0.000);
		
		\draw[gray] (x) -- (y) -- (z) --cycle;
		\node[below left] at (x) {$x$}; 
		\node[below right] at (y) {$y$};
		\node[above] at (z) {$z$};
		
		\coordinate (E1) at (0.531,  0.280);
		\coordinate (E2) at  (0.531,  0.227);
		\coordinate (E3) at  (0.592,  0.191);
		\coordinate (E4) at (0.670,  0.202);
		\coordinate (E5) at (0.702,  0.477);
		
		\draw[gray!40!white, fill = gray!30!white] (E1) -- (E2)--(E3)	--	(E4) --	(E5)  -- cycle; 
		\node at ($(E1)+(.07, 0)$) {\tiny $\cset$};
		
		\draw[very thin, dashed, blue!30!white] ($(E1)!-.8!(E2)$) -- ($(E1)!1.8!(E2)$);
		\draw[very thin, dashed, blue!30!white] ($(E2)!-.8!(E3)$) -- ($(E2)!1.8!(E3)$);
		\draw[very thin, dashed, blue!30!white] ($(E3)!-.8!(E4)$) -- ($(E3)!1.8!(E4)$);
		\draw[very thin, dashed, blue!30!white] ($(E4)!-.3!(E5)$) -- ($(E4)!1.3!(E5)$);
		\draw[very thin, dashed, blue!30!white] ($(E5)!-.3!(E1)$) -- ($(E5)!1.3!(E1)$);
		
		\coordinate (f12) at (0.100,  -0.000);
		\coordinate (f23) at (0.050,  0.087);
		\coordinate (f34) at (-0.014,  0.099);
		\coordinate (f45) at (-0.099,  0.011);
		\coordinate (f51) at (0.076,  -0.065);

		\draw[->, very thin] ($(E1)!.5!(E2)$) --  ++($-1*(f12)$) node[above] {\tiny $f_1$};
		\draw[->, very thin] ($(E2)!.5!(E3)$) --  ++($-1*(f23)$) node[below] {\tiny $f_2$};
		\draw[->, very thin] ($(E3)!.5!(E4)$) --  ++($-1*(f34)$) node[below] {\tiny $f_3$};
		\draw[->, very thin] ($(E4)!.5!(E5)$) --  ++($-1*(f45)$) node[above] {\tiny $f_4$};
		\draw[->, very thin] ($(E5)!.5!(E1)$) --  ++($-1*(f51)$) node[above] {\tiny $f_5$};
		\end{tikzpicture}		
		\caption{Credal set from Example~\ref{ex-lpr1} as an intersection of half planes: their support lines are dashed, gambles $f_i\in \mathcal K$ are depicted as normal vectors to faces }\label{fig-credal-half-planes}
	\end{minipage}\hfill
	\begin{minipage}[t]{0.45\textwidth}
		\centering
		\begin{tikzpicture}[scale = 8]
		\draw[white] (.3, 0) -- (.5, 0);
		
		\coordinate (E1) at (0.531,  0.280);
		\coordinate (E2) at  (0.531,  0.227);
		\coordinate (E3) at  (0.592,  0.191);
		\coordinate (E4) at (0.670,  0.202);
		\coordinate (E5) at (0.702,  0.477);
		
		\draw[gray!40!white, fill = gray!30!white] (E1) -- (E2)--(E3)	--	(E4) --	(E5)  -- cycle; 
		\node at ($(E1)+(.09, 0.03)$) {\small $\cset$};

		\coordinate (f12) at (0.100,  -0.000);
		\coordinate (f23) at (0.050,  0.087);
		\coordinate (f34) at (-0.014,  0.099);
		\coordinate (f45) at (-0.099,  0.011);
		\coordinate (f51) at (0.076,  -0.065);

		\fill[left color=gray!15!white, right color=gray!35!white!80!red, draw=white] (E1) --  ++($-1*(f12)$) -- ($(E1) - (f51)$) -- cycle;
		\draw[very thin, ->] (E1) --  ++($-1.1*(f12)$);
		\draw[very thin, ->] (E1) --  ++($-1.1*(f51)$);
		\node at ($(E1)+(-.1, .02)$) {\tiny $N_\mathcal M(E_1)$};
		
		\fill[top color=gray!15!white, bottom color=gray!35!white!80!red, draw=white] (E5) --  ++($-1.1*(f51)$) -- ($(E5) - 1.1*(f45)$) -- cycle;
		\draw[very thin, ->] (E5) --  ++($-1.1*(f51)$);
		\draw[very thin, ->] (E5) --  ++($-1.1*(f45)$);
		\node at ($(E5)+(.05, .02)$) {\tiny $N_\mathcal M(E_5)$};
		
		\draw[fill = white] (E1) circle[radius=.1pt] node[right] {\tiny $E_1$};
		\draw[fill = white] (E2) circle[radius=.1pt] node[right] {\tiny $E_2$};
		\draw[fill = white] (E3) circle[radius=.1pt] node[below] {\tiny $E_3$};
		\draw[fill = white] (E4) circle[radius=.1pt] node[right] {\tiny $E_4$};
		\draw[fill = white] (E5) circle[radius=.1pt] node[below right] {\tiny $E_5$};
		
		\end{tikzpicture}
		\caption{Normal cones at extreme points are the positive hulls of the normal vectors of adjacent faces.} \label{fig-normal-cones}
	\end{minipage}
\end{figure}

We extend a credal set $\cset$ to the set of vectors
\begin{equation}\label{eq-m-hat}
	\hat\cset  = \{ p \colon p\cdot (f-\lpr(f)1_\states) \ge 0, \text{ for every } f\in\mathcal K \}, 
\end{equation}
which is a convex cone, with the \emph{basis} $\cset$. This means that every $p\in\hat{\cset}$ is of the form $p = \lambda P$ for some $\lambda\ge 0$ and $P\in\cset$. (Note that $p\ge 0$ follows from $\lpr(1_x) \ge 0$ for every $x\in\states$, which are assumed by Convention~\ref{conv-include-1x}.) 

Given a credal set $\cset$, the \emph{cone of desirable gambles} contains exactly those gambles in $\allgambles$ whose lower prevision is non-negative:
\begin{equation}\label{eq-cone-desirable-gambles}
\mathcal D = \{ f\in\allgambles \colon p\cdot f \ge 0 \text{ for every } p\in\cset  \}.
\end{equation}
The gambles $f$ with $\lpr(f)=0$ are sometimes called \emph{marginally desirable}. 

\subsection{Normal cones of credal sets}

Let
\begin{equation}\label{eq-nc-gruber}
\csete=\{ x\in\RR^n \colon Ax\le b\}, 
\end{equation}
where $A$ is an $m\times n$ matrix and $b\in \RR^m$ a vector, be a convex polyhedron and $x$ a point on its boundary. According to \cite{gruber:07CDG}, the \emph{normal cone} at $x$ is the set 
\begin{equation}
N_\csete (x) = \{ u\colon u\cdot y \le u\cdot x \text{ for all } y\in \csete\} = \{ u\colon u\cdot (y-x) \le 0 \text{ for all } y\in \csete\}. 
\end{equation}

In our case, let $\cset$ be a credal set defined by constraints of the form \eqref{eq-cset-constraints-first} and \eqref{eq-constraint-equality} and $E$ its boundary point. The normal cone of $\cset$ at $E$ is the set
\begin{equation}\label{eq-normal-cone}
N_{\cset}(E) = \{ f \colon E(f) \le P(f) \text{ for every }  P\in\cset \}.
\end{equation}
The normal cone is thus the set of gambles $f$ that satisfy $E(f) = \lpr(f)$.
\begin{prop}[\cite{gruber:07CDG}~Proposition~14.1.]\label{prop-gruber1.4}
	Let $\csete$ be a convex polyhedron defined as in \eqref{eq-nc-gruber} and $x$ its boundary point.
	Let $a_i\cdot x = b_i$ hold for exactly $i\in  I\subseteq \{1, \ldots, m\}$, where $a_i$ denotes $i$-th row of the matrix $A$.   
	Then 
	$ N_\csete (x) = \mathrm{posi}\,\{ a_i\colon i\in I \},  $
	where $\mathrm{posi}$ denotes the \emph{non-negative hull}.  
\end{prop}

\begin{rem}
	We will call the set of vectors $\{ a_i \colon i \in I\}$ the \emph{positive basis} of the normal cone $N_\csete(x)$. 
\end{rem}
\begin{cor}
	Let $\cset$ be a credal set defined by constraints \eqref{eq-cset-constraints-first} and \eqref{eq-constraint-equality}. Then the set of desirable gambles $\mathcal D$ corresponding to $\cset$ is the normal cone of $\hat{\cset}$ at $\mathbf{0}$ and we have that 
	\begin{equation}\label{eq-desirable-normal-cone}
		 \mathcal D = \mathrm{posi}\,\{f-\lpr(f)1_\states \colon f\in \mathcal K\}. 
	\end{equation}
\end{cor}
\begin{proof}
	The set $\hat{\cset}$ is a convex cone whose support hyperplanes are exactly the sets of the form $H_f = \{ p\colon p\cdot (f-\lpr(f)1_\states) = 0 \}$ for $f\in \mathcal K$, and the origin is exactly the intersection of all support hyperplanes: $\mathbf{0}\cdot (f-\lpr(f)1_\states) = 0$ for every $f\in \mathcal K$. We can therefore apply Proposition~\ref{prop-gruber1.4}.  
\end{proof}
\begin{rem}
	In \cite{augustin2014introduction}~Chapter~1, the set constructed as in \eqref{eq-desirable-normal-cone} is called the natural extension of the assessment $\mathcal K^+ = \{f-\lpr(f)\colon f\in \mathcal K\} $. The fact that the set of desirable gambles is the positive hull of marginally desirable assessment $\mathcal K^+$ can also be found in Chapter~2 of the mentioned book. In the above references the set of strictly positive gambles is included separately. Here, however, the positive gambles are exactly the positive combinations of the gambles $1_x$, which are included by Convention~\ref{conv-include-1x}.
\end{rem}
\begin{cor}\label{cor-normal-cone-decomposition}
	Let $\cset$ be a credal set defined by constraints \eqref{eq-cset-constraints-first} and \eqref{eq-constraint-equality}, $E\in \cset$ a linear prevision and $h$ a gamble such that $E(h) = \lpr(h)$. Suppose that $E(f_i) = 0$ for exactly $i\in I\subseteq \{ 1, \ldots, n \}$. Then there exist $\alpha_i\ge 0$ for every $i\in I$ and $\beta \in \RR$ so that 
	\begin{equation}\label{eq-positive-combination+constant}
	h = \sum_{i \in I} \alpha_i f_i + \beta 1_\states. 
	\end{equation}	
\end{cor}
\begin{proof}
	Let $h\in\allgambles$ be a gamble such that $E(h) = \lpr (h)$. Set $g = h-\lpr (h)$. Then, for every $p\in \hat{\cset}$ (see \eqref{eq-m-hat}), $p = \alpha P$ for some $P\in\cset$ and $\alpha \ge 0$. Therefore $p\cdot g = \alpha P\cdot g \ge 0 = E\cdot g$, whence $g\in N_{\hat{\cset}}(E)$. By Proposition~\ref{prop-gruber1.4}, $g = \sum_{i\in I}\alpha_i f_i$ for some positive constants $\alpha_i$. Hence $h = \sum_{i\in I}\alpha_i f_i + \lpr(h)1_\states$, which proves the proposition. 
\end{proof}
\section{The distance between coherent lower previsions}\label{s-dbclp}
\subsection{The definition of the distance}
Let $\lpr$ and $\lpr'$ be two coherent lower previsions on the set of all gambles $\allgambles$ on a finite set $\states$.
We define the distance
between $\lpr$ and $\lpr'$ as 
\begin{equation}
d(\lpr, \lpr') = \max_{f\in\allgambles} \frac{|\lpr(f)-\lpr'(f)|}{\| f\|}, f\neq 0,
\end{equation}
where the norm $\| f \| = \sqrt{f\cdot f}$ is the Euclidean norm in $\RR^{|\states|}$. Clearly, the following alternative definition is equivalent:
\begin{equation}
d(\lpr, \lpr') = \max_{\substack{f\in\allgambles\\ \| f \| = 1}} |\lpr(f)-\lpr'(f)|.
\end{equation}
It is readily verified that the above distance function induces a metric in the set of all lower previsions on $\allgambles$. 
In this section we will analyze the maximal possible distance between two coherent lower previsions that coincide on a finite set of gambles. 

Suppose that $\lpr$ is a lower prevision on $\allgambles$, and the only information about it are the values on a finite set of gambles $\mathcal K\subset \allgambles$. 
That is, $\lpr(f)$ are given for every $f\in\mathcal K$.
We denote the restriction of $\lpr$ to $\mathcal K$ by $\lpr_{\mathcal K}$.  
We also adopt Convention~\ref{conv-include-1x}. 
The natural extension $\low E$ is the minimal (or the least committal) extension of $\lpr_{\mathcal K}$. 
This implies that $\lpr(f)\ge \low{E}(f)$ for every $f\in\allgambles$. 
Therefore, given another extension $\lpr'$ of $\lpr_{\mathcal K}$, we have that 
\begin{equation}\label{eq-maximal-dist-on-gamble}
|\lpr(f)-\lpr'(f)| \le \max\{ \lpr(f)-\low E(f), \lpr'(f)-\low E(f) \}, 
\end{equation}
which implies that 
$d(\lpr, \lpr') \le \max\{ d(\lpr, \low E), d(\lpr', \low E) \}.$
As we are interested in the maximal possible distance between coherent lower previsions coinciding on $\mathcal K$, it will therefore be enough to focus to the case where one of them is the natural extension of $\lpr_{\mathcal K}$. 

In the literature, other possible distances between coherent lower previsions or credal sets are used as well. In \cite{skulj:hable:13MET}, the total variation distance and the Hausdorff distance are used to explore coefficients of ergodicity for imprecise Markov chains. In \cite{Bronstein2008} various other metrics are used for measuring the distances between convex sets and their approximations. The choice of the metric used in a particular case depends on its interpretation, and as far as credal sets are concerned, the Euclidean metric seems to be the most appropriate for our case, because of its relation to the relative distances of the values of the corresponding lower previsions on gambles.  

\subsection{Maximal distance to the natural extension}
Let $\low E$ and $\lpr$ be respectively the natural extension of $\lpr_{\mathcal K}$ and another extension, and  $\cset$ and $\csete$ respectively their credal sets. 
As described in previous sections, both are convex sets and the natural extension is a convex polyhedron with extreme points $\extp(\cset)$. 

Assuming the above notations, we start with the following proposition.
\begin{prop}\label{prop-face-intersection}
	Take some $f\in\mathcal K$ and let $\face_f$ be the corresponding face of $\cset$. Then $\csete\cap \face_f \neq \emptyset$.  	
\end{prop} 
\begin{proof}
	Clearly, $\face_f$ contains exactly all linear previsions $P$ in $\cset$ such that $P(f) = \low P(f)$. 
	If no $P\in \csete$ belongs to $\face_f$, this then implies that $P(f) > \lpr(f)$ for every $P\in \csete$, and since $\csete$ is compact, this would imply that $\min_{P\in \csete} P (f) > \lpr(f)$, which contradicts the assumptions. 
\end{proof}
\begin{cor}\label{cor-5}
	Let $h\in \allgambles$ be an arbitrary gamble.  
	Then:
	\begin{enumerate}[(i)]
		\item $\lpr (h)\le \max_{P\in \face_f} P(h)   ~\text{for every}~ f\in\mathcal K; $ \label{cor-maximum-face}
		\item $\lpr(h)  \le \min_{f\in \mathcal K} \max_{P\in \cset_f} P(h)$; the inequality is tight in the sense that for every $h\in\allgambles$ an extension of $\lpr_{\mathcal K}$ exists that gives equality in the equation.  \label{eq-lpr-dominated-by-face}
		\item $\lpr(h) \le \min_{f\in\mathcal K} \max_{E\in\extp_f} E(h)$ where $\extp_f$ is the set of extreme points of the face $\face_f$; and the inequality is again tight. \label{cor-upper-bound-ext-points}
	\end{enumerate}
	
\end{cor}
\begin{proof}
	(\ref{cor-maximum-face}) is an immediate consequence of Proposition~\ref{prop-face-intersection}. 
	
	The inequality in (\ref{eq-lpr-dominated-by-face}) is a direct consequence of (\ref{cor-maximum-face}). It remains to prove that there is an extension of $\lpr_{\mathcal K}$ where the equality is reached. 
	
	Let $\face_f$ be a face of $\cset$ and let $P_f \in \arg\max_{P\in\face_f} P(h)$. 
	Let $\cset'$ be the convex hull of $\{P_f\colon f\in \mathcal K \}$ and $\lpr'$ the corresponding coherent lower prevision, which coincides with $\lpr$ on $\mathcal K$ by construction, and thus must satisfy the inequality (ii).  
	For every $P\in\cset'$, on the other hand, we have that $P=\sum_{f\in \mathcal K} \alpha_f P_f$, for some collection of values $\alpha_f\ge 0$ for every $f\in \mathcal K$ and $\sum_{f\in \mathcal K}\alpha_f = 1$.  
	Thus, 
	\begin{equation}
	P(h) = \sum_{f\in \mathcal K} \alpha_f P_f(h) 
	\ge \min_{f\in \mathcal K} P_f(h)
	= \min_{f\in \mathcal K} \max_{P\in \face_f}P(h) 
	\end{equation}
	Hence, $\low P'(h) = \min_{P\in \cset'} P(h) \ge \min_{f\in \mathcal K} \max_{P\in \face_f}P(h)$, which combined with the above reverse inequality gives the required equality. 
	
	The fact that extremal values are reached in extreme points easily implies (\ref{cor-upper-bound-ext-points}).
\end{proof}

Now we can express the maximal possible distance between two arbitrary extensions of $\lpr_{\mathcal K}$ in terms of its natural extension alone. 
\begin{cor}
	Let $\low E$ be the natural extension and $\lpr$ and $\lpr'$ two other extensions of $\lpr_{\mathcal K}$, and $h\in \allgambles$ a gamble. Then 
		\begin{equation}
	|\lpr(h)-\lpr'(h)| \le \min_{f\in\mathcal K} \max_{P\in\extp_f} P(h) - \low E(h)
		\end{equation}
	and
	\begin{equation}\label{eq-dist-next}
	d(\lpr, \lpr') \le \max_{\| h \| = 1} \min_{f\in\mathcal K} \max_{P\in\extp_f} P(h) - \low E(h).
	\end{equation}
\end{cor}
\begin{proof}
	The first inequality is a direct consequence of Corollary~\ref{cor-5}(\ref{cor-upper-bound-ext-points}) and Eq. \eqref{eq-maximal-dist-on-gamble}. The second inequality is an immediate consequence of the first one, definition of the distance between two coherent lower previsions and the fact that $\low E(h)$ is less than $P(h)$ for every feasible $P$. 
\end{proof}
Equation \eqref{eq-dist-next} gives the maximal possible distance between two unknown extensions of $\lpr_{\mathcal K}$ entirely in terms of its natural extension. From it we will derive an optimization methods providing computable upper bounds. 

By the definition of $\low E$ we have that $\low E(h) =  \min_{E\in\extp}E(h)$, whence we rewrite \eqref{eq-dist-next} into:
\begin{align}
d(\lpr, \lpr') & \le \max_{\| h \| = 1} \max_{E\in\extp} \min_{f\in\mathcal K} \max_{P\in\extp_f}  |P(h) - E(h)| \label{eq-maximizing-distance} \\
& = \max_{E\in\extp} \max_{\| h \| = 1}  \min_{f\in\mathcal K} \max_{P\in\extp_f}  |P(h) - E(h) \\
\intertext{by interchanging $\max_{\| h \| = 1}$ and $\min_{f\in \mathcal K}$ we obtain the inequality}
& \le  \max_{E\in\extp}\min_{f\in\mathcal K} \max_{P\in\extp_f}   \max_{\| h \| = 1} P(h) - E(h) \label{eq-maximizing-distance2}\\
& =  \max_{E\in\extp}\min_{f\in\mathcal K} \max_{P\in\extp_f}   d(P, E),
\end{align} 
where $d(P, E)$ is the Euclidean distance between extreme points $P$ and $E$. The last equality follows from the fact that $d(P, E) = \max_{\| h \| = 1} |P(h) - E(h)| = \max_{\| h \| = 1} \max\{ P(h) - E(h), E(h) - P(h) = P(-h) - E(-h) \}$ and since $\| -h \| = \| h \|$, the absolute value may be omitted. 

Now denote 
\begin{equation}\label{eq-bar-d}
\bar d(E, f) = \max_{P\in\extp_f}   d(P, E),
\end{equation}
which is the maximal Euclidean distance between an extreme point $E$ and a face $\cset_f$. Thus we obtain the following formula:
\begin{equation}\label{eq-upper-euclidean}
d(\lpr, \lpr') \le \max_{E\in\extp}\min_{f\in \mathcal K} \bar d(E, f).
\end{equation}
Since $E$ and $P$ in the above expressions are (extreme) points in $\RR^{|\states|}$, their Euclidean distances can be found easily by calculating the Euclidean norms $\|P-E\|$. Particularly, calculating $\bar d(E, f)$ requires calculating the Euclidean distances between $E$ and all extreme points of the face $\cset_f$. Finally, the RHS expression in \eqref{eq-upper-euclidean} is calculated by finding $\bar d(E, f)$ for all pairs of extreme points and gambles in $\mathcal K$. 

\subsection{Improved bounds}
Equation \eqref{eq-upper-euclidean} gives an upper bound for the difference between coherent lower previsions coinciding on a set of gambles, however, the estimate is systematically too conservative. This is caused by the fact that extreme points $E$ can only maximize expression \eqref{eq-maximizing-distance} for some $h$ if $E(h)=\low E(h)$. This means that the domain for $h$ in \eqref{eq-maximizing-distance2} should be restricted to those gambles $h$ that reach the lowest value $\low E(h)$ in $E$. In other words, $h$ should belong to the normal cone $N_\cset (E)$. 

Therefore, instead of taking the Euclidean distance between $E$ and $P$ in $\eqref{eq-bar-d}$, we should take the following distance:
\begin{equation}\label{eq-normed-distance}
d_E(E, P) = \max_{h\in N_\cset(E)}\frac{|P(h)-E(h)|}{\| h \|},
\end{equation} 
which we call the \emph{normed distance} between $E$ and $P$. 

The geometrical intuition behind replacing Euclidean distance with the above distance function is the following. Given a gamble $h$, the difference $P(h)-E(h)$ can be viewed as the inner product $(P-E)\cdot h$, which depends on the angle between $(P-E)$ and $h$. As the normal cone contains elements that are orthogonal to $P-E$ for  adjacent extreme points $P$, we may expect that the other elements are nearly orthogonal too, especially in the case of narrow normal cones. In Figure~\ref{fig-normal-cones} such situation can be observed in the case of the normal cone of $E_1$, in contrast to the case of $E_5$, where the normal cone is wide. Therefore, we would, for instance, expect that the normed distances between $E_1$ and its adjacent extreme points would be significantly smaller than the Euclidean distance, in contrast to the case of  $E_5$. Analytically we demonstrate this in Example~\ref{ex-dist}. 

In the sequel we represent the calculation of the normed distance in the form of a quadratic programming problem. 

\subsubsection*{Minimum norm elements of the normal cone.}
Consider an element $h$ of the form \eqref{eq-positive-combination+constant}. Given a pair of expectation functionals $E$ and $P$, the distance $P(h)-E(h)$ does not depend on $\beta$. In order to maximize the normed distance \eqref{eq-normed-distance}, we must consider the representative with the minimum norm, as the norm appears in the denominator of the expression. The characterization of the minimal norm element of the form \eqref{eq-positive-combination+constant} follows. 
\begin{prop}\label{prop-minimal-norm}
	Let $h$ be a gamble. Then $\| h + \beta 1_\states \| \ge \| h \|$ for every $\beta\in \RR$ if and only if $h\cdot 1_\states = 0$. 
\end{prop}
\begin{proof}
	We have that $\| h + \beta 1_\states\|^2 = \| h \| + \beta^2 + 2\beta h\cdot 1_\states$, which has minimum in $\beta = -h\cdot 1_\states$. Hence the minimizing $\beta$ equals 0 exactly if $h\cdot 1_\states$ does. 
\end{proof}
\begin{cor}\label{cor-ncone-positive-basis}
	Let $E, h$ and $I$ be as in Corollary~\ref{cor-normal-cone-decomposition} and let $f'_i$ be the unique vectors such that $f_i-f'_i=c1_\states$ and $f'_i\cdot 1_\states=0$ for every $i\in I$. Then, as follows from Corollary~\ref{cor-normal-cone-decomposition}, there exist some $\alpha'_i \ge 0$ for every $i\in I$ and $\beta'\in \RR$ so that
	\begin{equation}\label{eq-positive-cone-decomp1}
	h = \sum_{i\in I} \alpha'_i f'_i + \beta' 1_\states. 
	\end{equation} 
	Moreover, 
	\begin{equation}
	\left\| \sum_{i\in I} \alpha'_i f'_i \right\| \le \left\| \sum_{i\in I} \alpha'_i f'_i + \beta 1_\states \right\| ~\text{for every }\beta\in \RR.
	\end{equation}
\end{cor}
\begin{proof}
	Since $f'_i\cdot 1_\states = 0$, we have that $\left ( \sum_{i\in I} \alpha'_i f'_i \right) \cdot 1_\states = 0$, whence by Proposition~\ref{prop-minimal-norm} it follows that this is the minimal-norm gamble of the form \eqref{eq-positive-cone-decomp1}. 
\end{proof}
Let $I$ and $f'_i$, for $i\in I$, be as in Corollary~\ref{cor-ncone-positive-basis} and let $\underline{\alpha}\colon I\to [0, \infty)$ be a map and $\beta\in \RR$ a constant (we will write $\alpha_i$ instead of $\alpha(i)$). Then we define $h(\underline{\alpha}, \beta) = \sum_{i\in I} \alpha_if'_i + \beta 1_\states$. Clearly, $h(\low{\alpha}, \beta)\in N_\cset(E)$ and every element of $N_\cset(E)$ is of the form $h(\low{\alpha}, \beta)$, by Corollary~\ref{cor-normal-cone-decomposition}. 
\begin{cor}\label{cor-minimal-normed-distance}
	The following equality holds:
	\begin{equation}\label{eq-max-normed-distance}
	\max_{(\low{\alpha}, \beta)} \dfrac{| E(h(\low\alpha, \beta)) - P(h(\low\alpha, \beta)) |}{\|h(\low\alpha, \beta) \|} = \max_{\low\alpha} \dfrac{| E(h(\low\alpha, 0)) - P(h(\low\alpha, 0))|}{\|h(\low\alpha, 0) \|}
	\end{equation}
\end{cor}
\begin{proof}
	Since $| E(h+\beta 1_\states) - P(h+\beta 1_\states) | = | E(h) - P(h) |$, the maximum of the expression is achieved at $h$ with the minimum norm, which is the one with $\beta = 0$. 
\end{proof}
\subsubsection*{The calculation of the normed distance between expectation functionals.}
Take two linear expectation functionals $P$ and $E\in \cset$ and let $I$ and $f'_i$ for $i\in I$ be as in Corollary~\ref{cor-ncone-positive-basis}.  Our goal is to find the normed distance \eqref{eq-normed-distance}.
The absolute value in the numerator of \eqref{eq-normed-distance} can be omitted because $E(h) =\min_{P\in\cset} P(h)$ for every $h\in N_\cset(E)$. 
By Corollary~\ref{cor-minimal-normed-distance}, every $h\in N_\cset(E)$ that can minimize the above expression is of the form $h(\low{\alpha}, 0)$. Since $E$ and $P$ are themselves vectors too, we can denote $D = P-E$, and write
$ P(h) - E(h) = (P-E)\cdot h = D\cdot h. $

Now we can decompose every $f'_i$ for $i\in I$ as
$f'_i = \lambda_i D + u_i,$
so that $D\cdot u_i = 0$. 
Given that $h = \sum_{i\in I}\alpha_if'_i$, we obtain 
$h = (\low{\alpha}\cdot\low{\lambda})D + \low{\alpha}\cdot U,$
where $U$ is the matrix whose rows are $u_i, \low \lambda$ is the column vector with components $\lambda_i$ and the vectors $f'_i$ are also written as row vectors. We also assume $\low{\alpha}$ to be a column vector.

Further we have that
$\| h \|^2 = h\cdot h = \| D \|^2 \low{\alpha}\, \low{\lambda}\, \low{\lambda}^t \low{\alpha}^t + \low{\alpha}U U^t \low{\alpha}^t.$
Now denote $\Pi = \| D \|^2\low{\lambda}\, \low{\lambda}^t + U U^t$ and write
$\| h \|^2 =  \low{\alpha} \Pi \low{\alpha}^t. $
Clearly, $\Pi$ is a symmetric and positive semi-definite matrix. 

Moreover, we have that 
$P(h)-E(h) = D\cdot (\low{\alpha}\cdot \low{\lambda}) D = (\low{\alpha}\cdot \low{\lambda}) \| D \|^2.$
Our goal is the maximization of expression \eqref{eq-normed-distance}. Thus we need to maximize 
\begin{equation}\label{eq-maximal-normed-distance-explicit}
\varphi(\low{\alpha}) = \frac{ (\low{\alpha}\cdot \low{\lambda}) \| D \|^2}{\sqrt{\low{\alpha} \Pi \low{\alpha}^t}}
\end{equation}
over the set of all $I$-vectors $\low{\alpha}$ with non-negative components.
Clearly, for every non negative constant $k$ we have that $\varphi(k\low{\alpha}) = \varphi(\low{\alpha})$. Moreover, only those $\low\alpha$ for which the numerator in $\varphi (\low{\alpha})$ is positive are of interest, and then multiplying $\low{\alpha}$ by a suitable positive constant can ensure that the numerator is 1. Maximizing $\varphi(\low{\alpha})$ is then equivalent to minimizing the nominator, which yields the following quadratic programming problem:
\begin{quote}
	Minimize:
	\begin{align}
	\low{\alpha} \Pi \low{\alpha}^t \label{eq-minimizing-nc-objective}
	\intertext{subject to }
	(\low{\alpha}\cdot \low{\lambda}) \| D \|^2 = 1 \label{eq-minimizing-nc-constraints}\\
	\low{\alpha} \ge 0  \label{eq-minimizing-nc-nonnegativity}
	\end{align}	
\end{quote}

\begin{ex}	\label{ex-dist}
	Consider the lower prevision $\lpr$ from Example~\ref{ex-lpr1}. We will calculate the distance $d_{E_1}(E_1, E_5)$, where  $E_1 = (0.4,	0.32, 	0.28)$ and $E_5 = (0.15,	0.37,	0.48	) $.
	First we have: 
	\[ D= E_5-E_1 = (-0.2462, 0.0492, 0.1969) , \] 
	and its norm, which is the Euclidean distance between the two extreme points is $\| D \|  = 0.3191$. 
	The positive basis of $N_\cset (E_1)$ consists of the transformed gambles 
	\begin{align*}
	f'_1 & = f_1 - f_1\cdot 1_\states/3 = (-0.5, 0.5, 0)  \\
	f'_5 &= f_5 - f_5\cdot 1_\states/3 = (-0.2, 0.6, -0.4). 
	\end{align*}
	(see Corollary~\ref{cor-ncone-positive-basis}). 
	
	We have $f'_1 = 1.451D + (-0.1429, 0.4286, -0.2857)$, and since $f'_5$ is orthogonal to $D$, it follows that $u_5 = f'_5$ and $\lambda_2 = 0$. Thus
	$ \low{\lambda} = \begin{bmatrix}
	1.451 \\
	0 \\	
	\end{bmatrix}$ and
	$ 
	U = \begin{bmatrix}
	-0.14&	0.43 &	-0.29 \\
	-0.20&	0.60 &	-0.40 \\	
	\end{bmatrix} $
	which gives
	$ \Pi = \| D \|^2\low{\lambda}\, \low{\lambda}^t + U U^t = \begin{bmatrix}
	0.5 &	0.4  \\
	0.4&	0.56 \\	
	\end{bmatrix}. $
	Taking $\low{\alpha} = (\alpha_1, \alpha_2)^t$, we obtain the objective function to be minimized:
	$ \low{\alpha}\Pi\low{\alpha}^t = 0.5 \alpha_1^2+0.8 \alpha_1 \alpha_2+0.56 \alpha_2^2 $
	subject to 
	$ \| D\|^2\low{\alpha}\cdot \low{\lambda} = \| D\|^2 \lambda_1 \alpha_1 = 1 $
	whence $\alpha_1 = 6.7708$. Substituting $\alpha_1$ in the objective function we obtain
	$ \low{\alpha}\Pi\low{\alpha}^t = 22.9219 + 5.41664 \alpha_2 + 0.56 \alpha_2^2, $
	which has to be minimized subject to $\alpha_2 \ge 0$. The minimum is obtained for $\alpha_2 = 0$, with the minimal value of objective function $\low{\alpha}\Pi\low{\alpha}^t$ equal to $22.9219$. 
	Now 
	$ d_{E_1}(E_1, E_5) = \varphi(\low{\alpha}) = 1/\sqrt{22.9219} = 0.2089. $
	Note that this is significantly less than the Euclidean distance between the points, which is equal to $\| D \| =  0.3191$.
\end{ex}
The results in the example show that the maximal normed distance is reached on the gamble that makes the smallest angle with $D$ among all gambles in the normal cone, and in our case this is obviously $f'_1$, since adding any positive part of $f'_5$, which is perpendicular to $D$, would only increase the angle. The normed distance is in general clearly bounded by the Euclidean distance, which is reached only in the case where the normal cone contains a vector that is parallel with $D$. In most cases, however, particularly when the normal cones are narrow, the normed distance is typically significantly lower than the Euclidean distance. 

In the next section we provide an algorithm for calculating the bound on the distance between coherent extensions of a coherent lower prevision $\lpr_{\mathcal K}$. It calculates the value of 
\begin{equation}\label{eq-final-bound}
  d_{\max} = \max_{E\in\extp}\min_{f\in\mathcal K} \max_{P\in\extp_f}   d_E(E, P),
\end{equation}
which, based on the inequality \eqref{eq-maximizing-distance2}, bounds the maximal distance. To calculate the above bound, we need to consider every extreme point of $\cset$ and calculate the normed distance to the extreme points of all faces $\cset_f$ of $\cset$. The goal being finding the face whose most distant extreme point from $E$ is nearest to it. This obviously requires a lot of redundant analysis. The following easy criterion allows reducing the number of optimization steps substantially.  
\begin{prop}\label{prop-exclude-extreme}
	Let $E$, $F$ and $F'$ be linear previsions and $h\in N_\cset (E)$. Suppose that $F'(f_i) \ge F(f_i)$ for all the elements of the positive basis of $N_\cset(E)$. Then $F'(h)\ge F(h)$ for every $h\in N_\cset(E)$. 
\end{prop}
\begin{proof}
	An easy consequence of the fact that every element $h\in N_\cset(E)$ is a positive combination of elements $f_i$ contained in the positive basis of $N_\cset(E)$. 
\end{proof}
In the circumstances described by the above proposition we will say that an extreme point $F'$ \emph{dominates} $F$ on $N_\cset(E)$. If an extreme point $F'$ dominates $F$, then $d_E(E, F') \ge d_E(E, F)$. 

\section{Algorithm}\label{s-a}
\subsection{Outline}
The algorithm for finding the maximal distance between a lower prevision and the natural extension of its restriction to $\mathcal K$ is based on equation \eqref{eq-final-bound}. As shown in previous sections, the maximal distance can be computed in terms of extreme points $\extp$ of $\cset$. Efficient algorithms for finding the extreme points are known (\cite{dyer1977algorithm, DYER1982359, matheiss1980survey, chen1991line, Barber:1996:QAC:235815.235821}), whose worst case complexity is estimated to $O(n^2dv)$ (see e.g. \cite{Avis1992}), where $n$ is the number of constraints, $d$ the dimension, and $v$ the number of extreme points (vertices). 

For every extreme point $E$ we need to find the face whose most distant point is nearest to $E$. It is reasonable to start with the faces nearest to $E$, which certainly are those, whose extreme points include $E$. Finding the maximal distances for those faces gives a reasonable estimate of the maximal distance for the given $E$; however, there might exist faces whose most distant points are nearer than that. Yet, examining all faces would mean a lot of redundant analysis. In fact, in most cases examining the neighbour faces gives reasonable estimate, which is seldom improved by analysing the remaining faces.  

Proposition~\ref{prop-exclude-extreme} gives a useful criterion for filtering out faces that are too distant from $E$, and thus reducing the number of faces that need to be examined in cases where we want the exact value of \eqref{eq-final-bound}. That is, a face $\mathcal M_{f'}$ is filtered out if there exists an already analyzed face $\mathcal M_{f}$, so that every extreme point of $\mathcal M_{f'}$ dominates all extreme points of $\mathcal M_{f}$. In that case the maximal distance from $E$ to $\mathcal M_{f'}$ clearly cannot be smaller than the maximal distance to $\mathcal M_{f}$. And since we are looking for the minimal maximal distance, $\mathcal M_{f'}$ can be left out. Filtering out such faces is relatively fast operation that significantly reduces the number of quadratic optimizations needed. 

Applied in another way, Proposition~\ref{prop-exclude-extreme} also allows reducing the number of calculations of the distance for the extreme points within a face. If an extreme point, whose distance to $E$ has already been calculated, dominates some other extreme points in the same face, they can clearly be left out of calculations because they cannot produce a larger distance from $E$. 
\subsection{Parts of the algorithm}

\subsubsection*{Complete constraints} This part adds the non-negativity constraints of the form $p\cdot 1_x \ge 0$ and then removes the possible loose constraints. A constraint $p\cdot f \ge \lpr(f)$ is loose if there is no extreme point $E\in\extp$ such that $E(f) = \lpr(f)$. Note that except for the non-negativity constraints, coherence of the lower prevision in principle prevents the existence of loose constraints. The function {\sc RemoveRedundantConstraints}(\texttt{fn, lpr, EP}) returns \texttt{fn} and \texttt{lpr} inducing the same set of extreme points \texttt{EP} and without loose constraints. 
  
 \subsubsection*{Finding extreme points} This step applies one of the existing algorithms for finding extreme points of a convex polyhedron. The inputs of the function  {\sc GenerateExtremePoints}(\texttt{gmb, lpr}) are the set of gambles \texttt{gmb}, which is in the form of an $n\times s$ matrix, where $n$ is the number of gambles and $s$ is their dimension and \texttt{lpr} is the column vector of their lower previsions. The output is the set of extreme points in the form of a $v\times s$ matrix \texttt{V}. 

\subsubsection*{Finding the distance between the extreme points $E$ and $F$ (Algorithm~\ref{alg-nd})} Maximizes the expression \eqref{eq-normed-distance} on the normal cone $N_\cset(E)$. The problem translates to solving quadratic programming problem \eqref{eq-minimizing-nc-objective}--\eqref{eq-minimizing-nc-nonnegativity}. The inputs are the extreme points \texttt{E} and \texttt{P} as vectors of length $s$ and the gambles that form the basis of the normal cone of \texttt{E} in the form of $I\times s$ matrix \texttt{fpos}.
\begin{algorithm}
	\caption{Function: normed distance}\label{alg-nd}
	\begin{algorithmic}[1]
		\Function{NormedDistance}{\texttt{E, P, fpos}} 
			\State \texttt{D} $\gets$ \texttt{P-E}
			\State \texttt{nD} $\gets$ $\mathtt{\sqrt{D\cdot D}}$ \Comment{norm of \texttt{D}}
			\State set \texttt{Dmat} to be the matrix compatible with \texttt{fpos} wit all rows equal \texttt{D}
			\State $\mathtt{\lambda \gets fpos\cdot Dmat /nD}$ \Comment{$\mathtt{\lambda}$ becomes a column vector}
			\State $\mathtt{u \gets fpos - \lambda \cdot Dmat}$ \Comment{\texttt{u} becomes a matrix}
			\State $\mathtt{Pi \gets nD^2\lambda \lambda^t + uu^t}$
			\State $\displaystyle\mathtt{dist \gets \min \alpha^t \Pi \alpha}$ subject to $\mathtt{nD^2\lambda\cdot \alpha = 1, \alpha\ge 0}$
			\State \Comment{minimize over the set of $s$-dimensional vectors} 
			\State \Comment{using quadratic programming}
			\State \Return \texttt{dist}
		\EndFunction
		\end{algorithmic}
\end{algorithm}

\subsubsection*{Finding extreme points dominating/dominated-by an extreme point} According to Proposition~\ref{prop-exclude-extreme}, an extreme point that dominates another extreme point on the basis  of a normal cone, dominates it in the entire cone. This fact allows optimizing the algorithm, since several optimization steps are not needed in the case of dominance. The function is called \\\noindent 
{\sc DominatedExtremePoints($\mathtt{E}$, $\mathtt{points}$, $\mathtt{fpos}$)}. The inputs are an extreme point $\mathtt{E}$, a list of extreme points $\mathtt{points}$ and a list of gambles $\mathtt{fpos}$. The output is a set of indices of those extreme points from the list $\mathtt{points}$ that are dominated by $\mathtt{E}$. 

\subsubsection*{Filter faces} Filters the faces containing extreme points that dominate entire faces already analyzed. These faces can be left out of further analysis. Function {\sc FilterDominatingFaces($\mathtt{fcs}$, $\mathtt{domP}$)} returns those among faces \texttt{fcs}, whose set of extreme points does not intersect the set of dominating extreme points \texttt{domP}.

\subsubsection*{Find the bound on the distance (Algorithm~\ref{alg-fmd})} Finds the maximum of \eqref{eq-final-bound}. As the inputs we take a set of constraints in the form of gambles \texttt{gmb} and their lower previsions \texttt{lpr}. The function returns the maximal possible distance between any two extensions of these assessments to the set of all gambles. 
\begin{algorithm}
	\caption{Function: find maximal distance}\label{alg-fmd}
	\begin{algorithmic}[1]
		\Function{MaximalDistance}{\texttt{gmb, lpr}} 
			\State \texttt{V} $\gets$\ GenerateExtremePoints(\texttt{fn, lpr})
			\State  \texttt{maxDist = $0$} 
			\For {each \texttt{E} in \texttt{V}} 
				\State  \texttt{minDist $\gets$\ $\infty$}
				\State \texttt{fpos} $\gets$\ $\{ \mathtt{f}\in \mathtt{gmb \colon E\cdot f = lpr(f) \}}$ 
				\State \Comment{get support gambles of faces whose extreme point is \texttt{E}}
				\State \Comment{they constitute positive basis of the normal cone}	
				\For {each $\mathtt{f\in fpos\backslash \texttt{added non-negativity constraints}}$} 
					\State \texttt{Vf} $\gets$\ $\{ \mathtt{P}\in \mathtt{V \colon E\cdot f = lpr(f) \}}$ \label{alg-start-block}
					\State \Comment{get all extreme points of the face $\cset_f$}
					\State \texttt{Vf}$\gets$ SortByEuclidianDistanceToE(Vf) 
					\State \Comment{we start with the point that is most distant from $\mathtt{E}$}
					\State  \texttt{maxFaceDist $\gets$\ 0} 
					\State \Comment{maximal distance to an extreme point of the current face}
					\State \texttt{dominated} $\gets \emptyset$ 
					\For {each $\mathtt{P\in Vf}$}  
						\If {\texttt{P $\in$\ dominated} }			
							\State \texttt{d $\gets$\ maxFaceDist} 
							\State \Comment{distance calculation on} 
							\State \Comment{dominated faces is unnecessary}
						\Else
							\State \texttt{d} $\gets$\ NormedDistance(\texttt{E, P, fpos}) \label{alg-nd-call}
						\EndIf
						\State \texttt{maxFaceDist} $\gets$\ max(maxFaceDist, d)
						\State \texttt{VfD}$\gets$ DominatedExtremePoints(\texttt{P, Vf, fpos}) \label{alg-filter-points-start}
						\State \texttt{dominated $\gets$\ dominated $\cup$\ VfD} 
						\State \Comment{exclude all points dominated by \texttt{P} on the normal cone}	\label{alg-filter-points-end}									
					\EndFor					
					\State \texttt{minDist $\gets$\ min(minDist, maxFaceDist)}	\label{alg-end-block}
					\State \texttt{filtGmb $\gets$} FilterDominatingFaces(\texttt{gmb, tested})
					\State \Comment{filter gambles dominating already tested points} \label{alg-filter-faces-start}
					\State repeat steps \ref{alg-start-block}--\ref{alg-end-block} with \texttt{filtGmb} in place of \texttt{Vf} \label{alg-filter-faces-end}
				\EndFor				
				\State \texttt{maxDist $\gets$\ max(maxDist, minDist)}
			\EndFor
			\State \Return \texttt{maxDist}
		\EndFunction
	\end{algorithmic}
\end{algorithm}
In principle, our algorithm calculates the distances between all extreme points that lie in a common face $\cset_f$ for some $f\in \mathcal K$. After that it also checks all other unfiltered faces. This part is meant to ensure that the calculated bound is exact, although, such faces rarely exist or even improve the calculated distances. If only a good estimate is needed, this step may as well be omitted. 

Excluding dominated points within a face (lines \ref{alg-filter-points-start}--\ref{alg-filter-points-end}) significantly improves the efficiency of the algorithm. Empirical testing shows that only distances for a fraction of points then need to be calculated. Since the calculation of the distances is by far the slowest part of the algorithm, this significantly shortens the run-time. 
\subsection{Complexity estimation of the algorithm}
 Space complexity is determined by the number of extreme points, which depends on the shape of the credal set. For general lower previsions, even in low dimensions their number can be arbitrarily large; however, in the case of lower-upper probability pairs the upper bound for the number of extreme points is reported to be $s!$, where $s$ is equal to the number of elements of the probability space (see~\cite{WALLNER2007339}). Special classes of coherent lower-upper probability pairs have also been analyzed with the focus to their extreme points in \cite{miranda2003extreme, Miranda201544}. 

A more severe obstacle than space complexity is its time complexity. By far the slowest part of our algorithm is the calculation of the distance between extreme points (Algorithm~\ref{alg-fmd}, line \ref{alg-nd-call}) described in Algorithm~\ref{alg-nd}. Our complexity analysis will therefore focus on the number of calls to this routine. The time complexity of the routine is around $o(s^3)$, but since $s$ will typically be small compared to other variables, we may regard it as constant, and do the time complexity analysis based on other factors. The most important factor is certainly the number of extreme points, which we will denote by $v$; the number of constraints, which roughly coincides with the number of facets, will be denoted by $n$ and the dimension by $d$, typically $d=s-1$. The time complexity of the enumeration of extreme points of polyhedra is $O(n^2dv)$ (see \cite{Avis1992}), where $n$ is the number of constraints, $d$ the dimension, which is typically equal to $s-1$, and $v$ is the number of vertices.  

Typically, an extreme point is a solution of a system of $d$ equations forming constraints. This means that every extreme point is most usually adjacent to $d$ facets. The number of vertices per facet is then on average equal to $\frac{vd}{n}$, and since the distance to all vertices of faces adjacent to a vertex must be calculated, this gives $\frac{vd^2}{n}$ vertices. Vertices that were counted twice, because they lie on two facets, must be subtracted from this number. The distance must be calculated for every vertex, so this number must be multiplied by $v$. Depending on the ratio $\frac{v}{n}$, the number of pairs for which the quadratic programming routine \texttt{NormedDistance} must be called is bounded above by $d^2v^2/n$. In practice, the number of calls is significantly reduced by eliminating dominated extreme points. 

The time complexity is therefore exponential as a function of $s$ and polynomial as a function of $v$. Empirical testing shows that even for relatively low dimensional cases the algorithm's complexity is high. The complexity increases with the number of constraints, which is the number of gambles in $\mathcal K$. Notice however that with the size of $|\mathcal K|$ the accuracy of the approximation of the partially specified lower prevision increases. Therefore, the maximal possible distance, which is the maximal possible error of the approximation, is more important in the cases where the number of estimates is low; and in those cases the computational complexity of the algorithm is lower. 


\subsection{Numerical testing}
For numerical testing we implemented the algorithm on a sample of 10 randomly generated lower probabilities on probability spaces of sizes $s=|\states|=3, 4, 5, 6$ and 7. The constraints are formed by the lower probabilities of non-trivial subsets, whose number is $2^s-2$; the number of extreme points was in general close to the maximal possible number, which is $s!$. According to the complexity estimation from the previous section the upper bound for the number of calls without filtering the dominated extreme points would be 
\begin{equation}
	\frac{(s-1)^2(s!)^2}{2^s-2}
\end{equation}
plus the number of extreme points of non-dominated non-adjacent faces $\face_f$. The algorithm was designed to count the actual number of distances between extreme points that need to be considered (Table~\ref{table-results}, column 4), which slightly differs from the estimated value above. Further, Table~\ref{table-results} displays the average number of extreme points, which is also only slightly smaller than $s!$; the number of calls to the quadratic programming routine to calculate distances between pairs of extreme points; the number of distances that had to be considered (not all of them need actually be considered because the dominated ones can be quickly eliminated from the analysis); and the percentage of needed distances that actually had to be calculated using the quadratic programming routine. As expected, this percentage drops with the increase of dimension. 
\begin{table}[]
	\centering
	\begin{tabular}{c|rrrr}
		\toprule \\
		$|\states|$ & ext. pts & dist. calculated  & dist. needed  & ratio    \\\hline
		3         & 5.9                          & 11.8            & 11.8                                        & 100.00\% \\
		4         & 23.6                         & 124.2           & 300.4                                       & 41.34\%  \\
		5         & 101.2                        & 1697.2          & 6249.4                                      & 27.16\%  \\
		6         & 592.3                        & 31179.7         & 187453.2                                    & 16.63\%  \\
		7         & 2744.7                       & 586728.0          & 3911809.6                                   & 15.00\% \\
		\bottomrule
	\end{tabular}
	\caption{Test results by the sample space size: average number of extreme points; calls to quadratic programming routine; adjacent extreme points; ratio between the number of calls and the number of adjacent extreme points.}
	\label{table-results}
\end{table}

\section{Conclusions}\label{s-c}
The results obtained in this paper bound the maximal error of an approximation of a coherent lower prevision with its restriction to a finite set of gambles. This subsumes approximations of coherent lower previsions with more specific models, such as coherent lower probabilities. Such approximations are very common in the applications of imprecise probabilities, and our results give a first attempt to evaluate their errors. 

We have also provided an algorithm which calculates the bound based on calculating the normed distances between extreme points of a credal set. Since the number of extreme points grows rapidly with the dimension of the probability space, the computational complexity of the algorithm is in general high. 

The high computational complexity of the algorithm presents an obstacle that might hinder its practical applicability. Therefore, a direction of further research is finding more efficient algorithms based on the insights from the theoretical part of the paper. Another related question that could be examined with the help of the concepts developed in this paper is how to choose the most optimal set of gambles to approximate an unknown coherent lower prevision with minimal error. 

There are several applications of credal sets where approximations are used and could therefore benefit from the results of this paper. One of such applications are credal networks. In \cite{antonucci10a} approximations of credal sets with lower probabilities are proposed in the case of credal networks. Their findings suggest that as long as decisions based on credal networks are concerned, the use of lower probability approximations perform very well. This conclusion is based on numerical tests, where the decisions based on completely specified credal sets are compared with those based on the approximations with lower probabilities. The benefit of the use of our algorithm would be that the performance of lower probabilities could be estimated without having the credal sets completely specified. 

Another area where our results could be useful are imprecise Markov chains. One of the main problems of their estimation is that the complexity of credal sets corresponding to distributions of the chains grows exponentially in terms of the number of extreme points. In addition, current algorithms only allow the estimation of the corresponding lower expectations on single gambles. This means that the estimation for another gamble requires solving the optimization problem from the beginning. Since every estimation is computationally costly, which is especially true for the continuous time case, it is only feasible to do a certain number of estimations. Therefore, the question of the accuracy of the induced imprecise probability models arises naturally. 

\section*{Acknowledgement}
The author acknowledges the financial support from the Slovenian Research Agency (research core funding No. P5-0168).

\bibliography{../../references/references_all}
\end{document}